\documentclass[reqno]{amsart}
\usepackage[english]{babel}
\usepackage{amscd,amssymb,amsmath,amsfonts,latexsym,amsthm}
\usepackage[T1]{fontenc}
\usepackage{inputenc}
\usepackage{graphicx}

\newtheorem{thm}{Theorem}[section]

\newtheorem{prop}[thm]{Proposition}
\newtheorem{defn}[thm]{Definition}

\numberwithin{equation}{section}

\begin{document}

\title{Pre-derivations and description of non-strongly nilpotent filiform Leibniz algebras }
\author{K.K. Abdurasulov, A.Kh. Khudoyberdiyev, M. Ladra and A.M. Sattarov}
\address{[A.Kh. Khudoyberdiyev]
National University of Uzbekistan, Institute of Mathematics
Academy of Sciences of Uzbekistan, Tashkent, 100170, Uzbekistan.}
\email{khabror@mail.ru}
\address{[K.K. Abdurasulov -- A.M. Sattarov] Institute of Mathematics Academy of Sciences of Uzbekistan, Tashkent, 100170, Uzbekistan.}
\email{abdurasulov0505@mail.ru --- saloberdi90@mail.ru}
\address{[M. Ladra] Department of Mathematics, Institute of Mathematics,  University of Santiago de Compostela, 15782, Spain.}
\email {manuel.ladra@usc.es}



\begin{abstract}
In this paper we investigate pre-derivations of filiform Leibniz
algebras. Recall that the set of filiform Leibniz algebras of
fixed dimension can be decomposed into three non-intersected families.
We describe the pre-derivation of filiform Leibniz algebras for the first and second families. We
found sufficient conditions under which filiform Leibniz algebras
are strongly nilpotent. Moreover, for the first and second
families, we give the description of characteristically nilpotent
algebras which are non-strongly nilpotent.
\end{abstract}

\maketitle


\section{Introduction}

The notion of Leibniz algebra was introduced in \cite{Lod} as
a non-antisymmetric generalization of Lie algebras. During the
last 25 years, the theory of Leibniz algebras has been actively
studied and many results of the theory of Lie algebras have been
extended to Leibniz algebras. Since the study of derivations and
automorphisms of a Lie algebra plays an essential role in the
structure theory of algebras, it is a natural question whether
the corresponding results for Lie algebras can be extended to more
general objects.

In 1955, Jacobson \cite{Jac2} proved that a Lie algebra over a
field of characteristic zero admitting a non-singular (invertible)
derivation is nilpotent. However, the related problem of whether the inverse of this
statement is correct or not, remained open until the work of Dixmier and
Lister \cite{Dix}, where an example of a nilpotent Lie algebra,
whose all derivations are nilpotent (and hence, singular), was
constructed. Such types of Lie algebras are called
characteristically nilpotent Lie algebras. The papers \cite{CaNu,
Kha1} and others are devoted to the investigation of
characteristically nilpotent Lie algebras.

An analogue of Jacobson's result for Leibniz algebras was
proved in \cite{LaRiTu}. Moreover, it was shown that, similarly to
the case of Lie algebras, the inverse of this statement does not
hold and the notion of characteristically nilpotent Lie algebra
can be extended to Leibniz algebras \cite{KhLO, Omi1}. The study
of derivations of Lie algebras led to the appearance of a natural
generalization: pre-derivations of Lie algebras \cite{Mul};
the pre-derivations arise as the Lie algebra corresponding to the group of the pre-automorphisms of a Lie algebra.
 A  pre-derivation of a Lie algebra is just a
derivation of the Lie triple system induced by it. Research on pre-derivations has been related to  Lie algebra degenerations, Lie triple systems
and bi-invariant semi-Riemannian metrics on Lie groups \cite{Burd}. In
\cite{Baj} it was proved that Jacobson's result is also true in
terms of pre-derivations. Similarly, likewise the example of Dixmier and
Lister, several examples of nilpotent Lie algebras whose
pre-derivations are nilpotent were presented in \cite{Burd}. Such
Lie algebras are called strongly nilpotent.

In \cite{Moens}, a generalized notion of derivations and
pre-derivations of Lie algebras is given. These derivations are defined as Leibniz-derivations of
order $k$, and Moens proved that a finite-dimensional Lie algebra over a field of characteristic zero is nilpotent if and only if it admits an invertible Leibniz-derivation. Further, an analogue of Moens’ result was showed for alternative, Jordan, Zinbiel, Malcev and Leibniz algebras \cite{FKhO,Kayg1,Kayg2}.

Since a Leibniz-derivation of order $3$ of a Leibniz algebra is a
pre-derivation, it is natural to define the notion of strongly
nilpotent Leibniz algebras. It should be noted that the class of
strongly nilpotent Leibniz algebras is a subclass of
characteristically nilpotent Leibniz algebras. Thus, one of the
approaches to the classification of nilpotent Leibniz algebras
considers a subclass of Leibniz algebras, in which any Leibniz derivation
of order $k$ is nilpotent and any algebra admits a non-nilpotent
Leibniz-derivation of order $k+1$. In the case of $k=1$ we have
the class of non-characteristically nilpotent Leibniz algebras.
Such class of filiform Leibniz algebras was classified in
\cite{KhLO}.

This paper is devoted to the study of algebras for the case $k=2$, i.e. the class of characteristically
nilpotent filiform Leibniz algebras which are non-strongly
nilpotent. It is known that the class of all filiform Leibniz
algebras is split into three non-intersected families
\cite{AyOm2,GoOm}, where one of the families contains filiform Lie
algebras and the other two families come out from naturally graded
non-Lie filiform Leibniz algebras. An isomorphism criterion for
these two families of filiform Leibniz algebras was given in
\cite{GoOm}.  Note that some classes of finite-dimensional
filiform Leibniz algebras and algebras up to dimension less than
$10$ were classified in \cite{AAOK, GoJiKh,OmRa,RaSo}.

In order to achieve our goal, we have organized the paper as
follows. In Section~\ref{S:prel}, we present necessary definitions
and results that will be used in the rest of the paper.
In Section~\ref{S1} we describe pre-derivations of filiform Leibniz algebras of the first and second families.
 Finally, in Section~\ref{S2}, we give the description of
 characteristically nilpotent filiform Leibniz
algebras which are non-strongly nilpotent.

Throughout the paper, all the spaces and algebras are assumed to be
finite-dimensional.

\section{Preliminaries}\label{S:prel}

In this section we give necessary definitions and preliminary
results.

\begin{defn} An algebra $(L,[-,-])$ over a field $F$ is called a (right) Leibniz algebra if for any $x,y,z\in L$, the so-called Leibniz identity
\[ \big[[x,y],z\big]=\big[[x,z],y\big]+\big[x,[y,z]\big] \] holds.
\end{defn}

Every Lie algebra is a Leibniz algebra, but the bracket in a
Leibniz algebra does not need to be skew-symmetric.

Note that a derivation of Leibniz algebra $L$ is a
 linear transformation, such that
\[d([x,y])=[d(x),y]+[x, d(y)],\]
 for any $x, y\in L$.

Pre-derivations of Leibniz algebras are a generalization of
derivations which defined as follows.

\begin{defn}
A linear transformation $P$ of the Leibniz algebra $L$ is called a
pre-derivation if for any $x, y, z\in L$,
\[P([[x,y],z])=[[P(x),y],z]+[[x,P(y)],z]+[[x,y],P(z)].\]
\end{defn}

For the given Leibniz algebra $L$ we consider the following central
lower series:
\[
L^1=L,\qquad L^{k+1}=[L^k,L^1] \qquad k \geq 1.
\]

\begin{defn} A Leibniz algebra $L$ is called
nilpotent if there exists  $s\in\mathbb N $ such that $L^s=0$.
\end{defn}

A nilpotent Leibniz algebra is called \emph{characteristically
nilpotent} if all its derivations are nilpotent. We say that a
Leibniz algebra is  \emph{strongly nilpotent} if any
pre-derivation is nilpotent.

Since any derivation of the Leibniz algebra is a pre-derivation,
it implies that a strongly nilpotent Leibniz algebra is
characteristically nilpotent. An example of characteristically
nilpotent, but non-strongly nilpotent Leibniz algebra could be
find in \cite{FKhO, KhLO, Omi1}.

\begin{defn} A Leibniz algebra $L$ is said to be filiform if
$\dim L^i=n-i$, where $n=\dim L$ and $2\leq i \leq n$.
\end{defn}

The following theorem decomposes  all $n$-dimensional filiform
Leibniz algebras into  three families.

\begin{thm}[\cite{AyOm2,GoOm}]
 Any $n$-dimensional complex filiform Leibniz algebra admits a basis $\{e_1, e_2, \dots, e_n\}$
 such that the table of multiplication of the algebra has one of the following forms:

$F_1(\alpha_4, \dots, \alpha_n,\theta)=\left\{\begin{array}{ll}
[e_1,e_1]=e_3, \\[1mm] [e_i,e_1]=e_{i+1},& 2\leq i \leq n-1,\\[1mm]
[e_1,e_2]=\sum\limits_{t=4}^{n-1}\alpha_te_t+\theta e_n,&\\[1mm]
[e_j,e_2]=\sum\limits_{t=j+2}^{n}\alpha_{t-j+2}e_t,& 2\leq j\leq n-2.\\[1mm]
\end{array}\right.$ \\[1mm]

$F_2(\beta_4, \dots, \beta_n,\gamma)=\left\{\begin{array}{ll}
[e_1,e_1]=e_{3}, \\[1mm]
[e_i,e_1]=e_{i+1}, & \  3\leq i \leq {n-1},\\[1mm]
[e_1,e_2]=\sum\limits_{t=4}^{n}\beta_te_{t}, \\[1mm]
[e_2,e_2]= \gamma e_{n},\\[1mm]
[e_j,e_2]=\sum\limits_{t=j+2}^{n}\beta_{t-j+2}e_t, & \ 3\leq j
\leq {n-2},
\end{array} \right.$ \\

$F_3(\theta_1,\theta_2,\theta_3)= \left\{\begin{array}{lll}
[e_i,e_1]=e_{i+1}, &
2\leq i \leq {n-1},\\[1mm]
[e_1,e_i]=-e_{i+1}, & 3\leq i \leq {n-1}, \\[1mm]
[e_1,e_1]=\theta_1e_n, &   \\[1mm]
[e_1,e_2]=-e_3+\theta_2e_n, & \\[1mm]
[e_2,e_2]=\theta_3e_n, &  \\[1mm]
[e_i,e_j]=-[e_j,e_i] \in \langle e_{i+j+1}, e_{i+j+2}, \dots ,
e_n\rangle, &
2\leq i < j \leq {n-1},\\[1mm]
[e_i,e_{n+1-i}]=-[e_{n+1-i},e_i]=\alpha (-1)^{i+1}e_n, & 2\leq
i\leq n-1,
\end{array} \right.$

\noindent where all omitted products are equal to zero and
$\alpha\in\{0,1\}$ for even $n$ and $\alpha=0$ for odd $n$.
\end{thm}

It is easy to see that algebras of the first and the second
families are non-Lie algebras. Note that if $(\theta_1, \theta_2,
\theta_3) = (0,0,0)$, then an algebra of the third class is a Lie
algebra.

Further we shall need the  notion of Catalan numbers. The
 $p$-th Catalan numbers were defined
in \cite{HiPe} by the formula  \[C^{p}_{n} = \frac {1} {(p-1)n+1}
\ \binom{pn}{n}\,.\]

It should be noted that for the  $p$-th Catalan numbers the
following equality is hold:
\begin{equation} \label{E:comb}
\sum\limits_{k=1}^{n} C_k^p C_{n-k}^p = \frac {2n}
{(p-1)n+p+1}C_{n+1}^p\,.
\end{equation}

\section{Pre-derivations of filiform Leibniz algebras}\label{S1}

In this section we give the description of pre-derivations of
filiform Leibniz algebras. First, we consider the filiform Leibniz
algebras from the first family.

Let $L$ be a $n$-dimensional filiform Leibniz algebra from the
family $F_1(\alpha_4, \dots, \alpha_n, \theta )$.
\begin{prop}\label{prop1}
The pre-derivations of the filiform Leibniz algebras from the
family $F_1(\alpha_4,\alpha_5,\ldots, \alpha_{n}, \theta)$ have
the following form:
\begin{align*}
P(e_1)&=\sum\limits_{t=1}^{n}a_te_t,\quad
P(e_2)=(a_1+a_2)e_2+\sum\limits_{t=3}^{n-2}a_te_t+b_{n-1}e_{n-1}+b_ne_n,
\quad P(e_3)=\sum\limits_{t=2}^{n}c_te_t,\\
P(e_{2i})&=((2i-1)a_1+a_2)e_{2i}+\sum\limits_{t=2i+1}^{n}(a_{t-2i+2}+(2i-2)a_2\alpha_{t-2i+3})e_t,\qquad \qquad \qquad  2\leq i\leq \left\lfloor\frac{n}{2}\right\rfloor,\\
P(e_{2i+1})&=c_2e_{2i}+((2i-2)a_1+c_3)e_{2i+1}+\sum\limits_{t=2i+2}^{n}(c_{t-2i+2}+(2i-2)a_2\alpha_{t-2i+2})e_t,\qquad 2\leq i\leq \left\lfloor\frac{n-1}{2}\right\rfloor,
\end{align*}
where $\lfloor a\rfloor$ is the integer part of the real number $a$ and
\begin{equation}\label{eq0}\left\{\begin{array}{ll}
(1+(-1)^n)c_2=0,\quad c_2\alpha_t=0,& 4\leq t\leq n-1, \\[1mm]
(a_1-a_2)\alpha_4=0,\quad  (3a_1-c_3)\alpha_4=0,\\[1mm]
\sum\limits_{t=3}^{k}(a_{2k-2t+3}-c_{2k-2t+4}+a_2\alpha_{2k-2t+4})\alpha_{2t-2}=0,&3\leq k\leq \lfloor\frac{n-1}{2}\rfloor,\\[1mm]
(2a_1+a_2-c_3)\alpha_{2k}+\sum\limits_{t=3}^{k}(a_{2k-2t+4}-c_{2k-2t+5}+a_2\alpha_{2k-2t+5})\alpha_{2t-2}=0,& 3\leq k\leq \lfloor\frac{n}{2}\rfloor-1,\\[1mm]
(a_2-(k-3)a_1)\alpha_{k}=\frac{k-1}{2}a_2\sum\limits_{t=5}^{k}\alpha_{t-1}\alpha_{k-t+4},& 5\leq k\leq n-2,\\[1mm]
(a_2-(n-4)a_1)\alpha_{n-1}=a_2\sum\limits_{t=3}^{\frac{n-2}{2}}(2t-3)\alpha_{n-2t+3}\alpha_{2t-1}\\[1mm]
\quad\quad\quad\quad\quad\quad\quad\quad\quad
+\sum\limits_{t=2}^{\frac{n-2}{2}}(c_{n-2t+2}-a_{n-2t+1}+(2t-3)a_2\alpha_{n-2t+2})\alpha_{2t},&
n\quad  \text {is even},
\\[1mm]
(2a_2-c_3-(n-6)a_1)\alpha_{n-1}=a_2\sum\limits_{t=3}^{\frac{n-1}{2}}(2t-3)\alpha_{n-2t+3}\alpha_{2t-1}
+\\[1mm]
\quad\quad\quad\quad\quad\quad\quad\quad\quad
+\sum\limits_{t=2}^{\frac{n-3}{2}}(c_{n-2t+2}-a_{n-2t+1}+(2t-3)a_2\alpha_{n-2t+2})\alpha_{2t},&
n\quad \text {is odd.}
\end{array}\right.
\end{equation}
\end{prop}

\begin{proof}
Let $L$ be a filiform Leibniz algebra from the family
$F_1(\alpha_4,\alpha_5,\ldots, \alpha_{n}, \theta)$ and let $P \colon  L
\rightarrow L$ be a pre-derivation of $L$. Put
\[P(e_1)=\sum\limits_{t=1}^{n}a_te_t, \quad  P(e_2)=\sum\limits_{t=1}^{n}b_te_t, \quad P(e_3)=\sum\limits_{t=1}^{n}c_te_t.\]

From
\begin{align*}
0& =P([[e_1,e_1],e_3])=[[P(e_1),e_1],e_3]+[[e_1,P(e_1)],e_3]+[[e_1,e_1],P(e_3)]\\
& = [e_3, \sum\limits_{t=1}^{n}c_te_t] = c_1e_4 +
c_2\sum\limits_{t=4}^{n-1}\alpha_{t}e_{t+1},
\end{align*}
we have
\[ c_1=0,\qquad c_2\alpha_t=0,\quad 4\leq t\leq n-1.\]

By the property of pre-derivation, we have
\begin{align*}
P(e_4)&=P([[e_1,e_1],e_1])=[[P(e_1),e_1],e_1]+[[e_1,P(e_1)],e_1]+[[e_1,e_1],P(e_1)]\\
& =[(a_1+a_2)e_3+\sum\limits_{t=4}^{n}a_{t-1}e_t,e_1]+[a_1e_3+a_2\big(\sum\limits_{t=4}^{n-1}\alpha_te_t+\theta e_n\big),e_1]+a_1e_4+
a_2\sum\limits_{t=5}^{n}\alpha_{t-1}e_t\\
& =(3a_1+a_2)e_4+\sum\limits_{t=5}^{n}(a_{t-2}+2a_2\alpha_{t-1})e_t.
\end{align*}

On the other hand,
\begin{align*}
P(e_4)&=P([[e_2,e_1],e_1])=[[P(e_2),e_1],e_1]+[[e_2,P(e_1)],e_1]+[[e_2,e_1],P(e_1)] \\
&=[(b_1+b_2)e_3+\sum\limits_{t=4}^{n}b_{t-1}e_t,e_1]+[a_1e_3+a_2\sum\limits_{t=4}^{n}\alpha_te_t,e_1]+a_1e_4+
a_2\sum\limits_{t=5}^{n}\alpha_{t-1}e_t\\
&=(2a_1+b_1+b_2)e_4+\sum\limits_{t=5}^{n}(b_{t-2}+2a_2\alpha_{t-1})e_t.
\end{align*}

Comparing the coefficients at the basic elements we have
\[b_1+b_2=a_1+a_2, \qquad b_{t}=a_t,\quad 3\leq t\leq n-2.\]


Using the property of pre-derivation, we get
\begin{align*}
P(e_5)&=P([[e_3,e_1],e_1])=[[P(e_3),e_1],e_1]+[[e_3,P(e_1)],e_1]+[[e_3,e_1],P(e_1)]\\
&=[c_2e_3+\sum\limits_{t=4}^{n}c_{t-1}e_t,e_1]+[a_1e_4+a_2\sum\limits_{t=5}^{n}\alpha_{t-1}e_t,e_1]+a_1e_5+
a_2\sum\limits_{t=6}^{n}\alpha_{t-2}e_t\\
&=c_2e_4+(2a_1+c_3)e_5+\sum\limits_{t=6}^{n}(c_{t-2}+2a_2\alpha_{t-2})e_t.
\end{align*}

Similarly, from the equality
\[P(e_{j+2})=P([[e_j,e_1],e_1])=[[P(e_j),e_1],e_1]+[[e_j,P(e_1)],e_1]+[[e_j,e_1],P(e_1)],
\quad 3\le j\le n-2,\] inductively, we derive
\begin{align*}
P(e_{2i})&=((2i-1)a_1+a_2)e_{2i}+\sum\limits_{t=2i+1}^{n}(a_{t-2i+2}+(2i-2)a_2\alpha_{t-2i+3})e_t,&& 2\leq i\leq \left\lfloor\frac{n}{2}\right\rfloor.\\
P(e_{2i+1})&=c_2e_{2i}+((2i-2)a_1+c_3)e_{2i+1}+\sum\limits_{t=2i+2}^{n}(c_{t-2i+2}+(2i-2)a_2\alpha_{t-2i+2})e_t,&& 2\leq i\leq \left\lfloor\frac{n-1}{2}\right\rfloor.
\end{align*}

Moreover, in the case of $n$ is even using the property of the
pre-derivation for the triple $\{e_{n-1},e_1,e_1\}$, we deduce
$c_2=0$. In the case of $n$ is odd considering the  property of
$P([[e_{n-1},e_1],e_1])$ give us the identical equality. Thus we
get, \[(1 + (-1)^n)c_2 =0.\]

Consider
\begin{multline*}
  P([[e_1,e_1],e_2])=P([e_3,e_2])=P\Big(\sum\limits_{t=5}^{n}\alpha_{t-1}e_t\Big)\\
=\sum\limits_{t=2}^{\lfloor\frac{n-1}{2}\rfloor}\Big[c_2e_{2t}+((2t-2)a_1+c_3)e_{2t+1}+\sum\limits_{k=2t+2}^{n}(c_{k-2t+2}+(2t-2)a_2\alpha_{k-2t+2})e_k\Big]\alpha_{2t}\\
+\sum\limits_{t=3}^{\lfloor\frac{n}{2}\rfloor}\Big[((2t-1)a_1+a_2)e_{2t}+\sum\limits_{k=2t+1}^{n}(a_{k-2t+2}+(2t-2)a_2\alpha_{k-2t+3})e_k\Big]\alpha_{2t-1}\\
=\sum\limits_{k=2}^{\lfloor\frac{n-1}{2}\rfloor}((2k-2)a_1+c_3)\alpha_{2k}e_{2k+1}+\sum\limits_{k=6}^{n}\sum\limits_{t=2}^{\lfloor\frac{k-2}{2}\rfloor}(c_{k-2t+2}+(2t-2)a_2\alpha_{k-2t+2})\alpha_{2t}e_k\\
+\sum\limits_{k=3}^{\lfloor\frac{n}{2}\rfloor}((2k-1)a_1+a_2)\alpha_{2k-1}e_{2k}+\sum\limits_{k=7}^{n}\sum\limits_{t=3}^{\lfloor\frac{k-1}{2}\rfloor}(a_{k-2t+2}+(2t-2)a_2\alpha_{k-2t+3})\alpha_{2t-1}e_k\\
=(2a_1+c_3)\alpha_4e_5+\sum\limits_{k=3}^{\lfloor\frac{n}{2}\rfloor}\Big[((2k-1)a_1+a_2)\alpha_{2k-1}+\sum\limits_{t=2}^{k-1}(c_{2k-2t+2}+(2t-2)a_2\alpha_{2k-2t+2})\alpha_{2t}\\
+\sum\limits_{t=3}^{k-1}(a_{2k-2t+2}+(2t-2)a_2\alpha_{2k-2t+3})\alpha_{2t-1}\Big]e_{2k}\\
+\sum\limits_{k=3}^{\lfloor\frac{n-1}{2}\rfloor}\Big[((2k-2)a_1+c_3)\alpha_{2k}+\sum\limits_{t=2}^{k-1}(c_{2k-2t+3}+(2t-2)a_2\alpha_{2k-2t+3})\alpha_{2t}\\
+\sum\limits_{t=3}^{k}(a_{2k-2t+3}+(2t-2)a_2\alpha_{2k-2t+4})\alpha_{2t-1}\Big]e_{2k+1}.
\end{multline*}

On the other hand, using the property of pre-derivation, we get
\begin{multline*}
P([[e_1,e_1],e_2])=[[P(e_1),e_1],e_2]+[[e_1,P(e_1)],e_2]+[[e_1,e_1],P(e_2)]\\
=[(a_1+a_2)e_3+\sum\limits_{t=4}^{n}a_{t-1}e_t,e_2]+[a_1e_3+a_2\Big(\sum\limits_{t=4}^{n-1}\alpha_te_t+\theta e_n \Big),e_2]+b_1e_4+
b_2\sum\limits_{t=5}^{n}\alpha_{t-1}e_t\\
=(a_1+a_2)\sum\limits_{t=5}^{n}\alpha_{t-1}e_t+\sum\limits_{t=4}^{n-2}a_{t-1}\sum\limits_{k=t+2}^{n}\alpha_{k-t+2}e_k\\
+ a_1\sum\limits_{t=5}^{n}\alpha_{t-1}e_t+a_2\sum\limits_{t=4}^{n-2}\alpha_{t}\sum\limits_{k=t+2}^{n}\alpha_{k-t+2}e_k+b_1e_4+
b_2\sum\limits_{t=5}^{n}\alpha_{t-1}e_t\\
=b_1e_4+(2a_1+a_2+b_2)\alpha_{4}e_5+\sum\limits_{k=3}^{\lfloor\frac{n}{2}\rfloor}\Big[(2a_1+a_2+b_2)\alpha_{2k-1}+\sum\limits_{t=4}^{2k-2}(a_{t-1}+a_2\alpha_t)\alpha_{2k-t+2}\Big]e_{2k}\\
+\sum\limits_{k=3}^{\lfloor\frac{n-1}{2}\rfloor}\Big[(2a_1+a_2+b_2)\alpha_{2k}+\sum\limits_{t=4}^{2k-1}(a_{t-1}+a_2\alpha_t)\alpha_{2k-t+3}\Big]e_{2k+1}.
\end{multline*}

Comparing the coefficients at the basic elements we have
\begin{equation}\label{eq1}\left\{\begin{array}{cc}
b_1=0,\\[1mm] (2a_1+a_2+b_2)\alpha_4=(2a_1+c_3)\alpha_4,
\end{array}\right.
\end{equation}
\begin{equation}\label{eq2}3\leq k\leq \left\lfloor\frac{n}{2}\right\rfloor\left\{\begin{array}{cc}
(2a_1+a_2+b_2)\alpha_{2k-1}+\sum\limits_{t=4}^{2k-2}(a_{t-1}+a_2\alpha_t)\alpha_{2k-t+2}\\[1mm]
=((2k-1)a_1+a_2)\alpha_{2k-1}+\sum\limits_{t=2}^{k-1}(c_{2k-2t+2}+(2t-2)a_2\alpha_{2k-2t+2})\alpha_{2t}+\\[1mm]
+\sum\limits_{t=3}^{k-1}(a_{2k-2t+2}+(2t-2)a_2\alpha_{2k-2t+3})\alpha_{2t-1},
\end{array}\right.
\end{equation}
\begin{equation}\label{eq3}3\leq k\leq \left\lfloor\frac{n-1}{2}\right\rfloor\left\{\begin{array}{cc}
(2a_1+a_2+b_2)\alpha_{2k}+\sum\limits_{t=4}^{2k-1}(a_{t-1}+a_2\alpha_t)\alpha_{2k-t+3}\\[1mm]
=((2k-2)a_1+c_3)\alpha_{2k}+\sum\limits_{t=2}^{k-1}(c_{2k-2t+3}+(2t-2)a_2\alpha_{2k-2t+3})\alpha_{2t}+\\[1mm]
+\sum\limits_{t=3}^{k}(a_{2k-2t+3}+(2t-2)a_2\alpha_{2k-2t+4})\alpha_{2t-1}.
\end{array}\right.
\end{equation}

Now, we consider
\begin{multline*}
P([[e_3,e_1],e_2])=P([e_4,e_2])=P\Big(\sum\limits_{t=6}^{n}\alpha_{t-2}e_t\Big)\\
=\sum\limits_{t=3}^{\lfloor\frac{n-1}{2}\rfloor}\Big[((2t-2)a_1+c_3)e_{2t+1}+\sum\limits_{k=2t+2}^{n}(c_{k-2t+2}+(2t-2)a_2\alpha_{k-2t+2})e_k\Big]\alpha_{2t-1}\\
+\sum\limits_{t=3}^{\lfloor\frac{n}{2}\rfloor}\Big[((2t-1)a_1+a_2)e_{2t}+\sum\limits_{k=2t+1}^{n}(a_{k-2t+2}+(2t-2)a_2\alpha_{k-2t+3})e_k\Big]\alpha_{2t-2}\\
=\sum\limits_{k=3}^{\lfloor\frac{n-1}{2}\rfloor}((2k-2)a_1+c_3)\alpha_{2k-1}e_{2k+1}+\sum\limits_{k=8}^{n}\sum\limits_{t=3}^{\lfloor\frac{k-2}{2}\rfloor}(c_{k-2t+2}+(2t-2)a_2\alpha_{k-2t+2})\alpha_{2t-1}e_k\\
+\sum\limits_{k=3}^{\lfloor\frac{n}{2}\rfloor}((2k-1)a_1+a_2)\alpha_{2k-2}e_{2k}+\sum\limits_{k=7}^{n}\sum\limits_{t=3}^{\lfloor\frac{k-1}{2}\rfloor}(a_{k-2t+2}+(2t-2)a_2\alpha_{k-2t+3})\alpha_{2t-2}e_k\\
=(5a_1+a_2)\alpha_4e_6+\sum\limits_{k=3}^{\lfloor\frac{n-1}{2}\rfloor}\Big[((2k-2)a_1+c_3)\alpha_{2k-1}+\sum\limits_{t=3}^{k-1}(c_{2k-2t+3}+(2t-2)a_2\alpha_{2k-2t+3})\alpha_{2t-1} \displaybreak  \\
+\sum\limits_{t=3}^{k}(a_{2k-2t+3}+(2t-2)a_2\alpha_{2k-2t+4})\alpha_{2t-2}\Big]e_{2k+1}  \\
+\sum\limits_{k=4}^{\lfloor\frac{n}{2}\rfloor}\Big[((2k-1)a_1+a_2)\alpha_{2k-2}+\sum\limits_{t=3}^{k-1}(c_{2k-2t+2}+(2t-2)a_2\alpha_{2k-2t+2})\alpha_{2t-1}\\
+\sum\limits_{t=3}^{k-1}(a_{2k-2t+2}+(2t-2)a_2\alpha_{2k-2t+3})\alpha_{2t-2}\Big]e_{2k}.
\end{multline*}

On the other hand,
\begin{multline*}
P([[e_3,e_1],e_2])=[[P(e_3),e_1],e_2]+[[e_3,P(e_1)],e_2]+[[e_3,e_1],P(e_2)]\\
=[\sum\limits_{t=3}^{n}c_{t-1}e_t,e_2]+[a_1e_4+a_2\sum\limits_{t=5}^{n}\alpha_{t-1}e_t,e_2]+b_2\sum\limits_{t=6}^{n}\alpha_{t-2}e_t\\
=\sum\limits_{t=4}^{n-2}c_{t-1}\sum\limits_{k=t+2}^{n}\alpha_{k-t+2}e_k+
a_1\sum\limits_{t=6}^{n}\alpha_{t-2}e_t+a_2\sum\limits_{t=5}^{n-2}\alpha_{t-1}\sum\limits_{k=t+2}^{n}\alpha_{k-t+2}e_k+b_2\sum\limits_{t=6}^{n}\alpha_{t-2}e_t\\
=(2a_1+a_2+c_3)\alpha_{4}e_6+\sum\limits_{k=7}^{n}\Big[(2a_1+a_2+c_3)\alpha_{k-2}+\sum\limits_{t=5}^{k-2}(c_{t-1}+a_2\alpha_{t-1})\alpha_{k-t+2}\Big]e_k\\
=(2a_1+a_2+c_3)\alpha_{4}e_6+\sum\limits_{k=3}^{\lfloor\frac{n-1}{2}\rfloor}
\Big[(2a_1+a_2+c_3)\alpha_{2k-1}+\sum\limits_{t=5}^{2k-1}(c_{t-1}+a_2\alpha_{t-1})\alpha_{2k-t+3}\Big]e_{2k+1}\\
+\sum\limits_{k=4}^{\lfloor\frac{n}{2}\rfloor}\Big[(2a_1+a_2+c_3)\alpha_{2k-2}+\sum\limits_{t=5}^{2k-2}(c_{t-1}+a_2\alpha_{t-1})\alpha_{2k-t+2}\Big]e_{2k}.
\end{multline*}

Comparing the coefficients at the basic elements we get
\begin{equation}\label{eq4}
(2a_1+a_2+c_3)\alpha_{4}=(5a_1+a_2)\alpha_4
\end{equation}
\begin{equation}\label{eq5}3\leq k\leq \left\lfloor\frac{n-1}{2}\right\rfloor\left\{\begin{array}{cc}
(2a_1+a_2+c_3)\alpha_{2k-1}+\sum\limits_{t=5}^{2k-1}(c_{t-1}+a_2\alpha_{t-1})\alpha_{2k-t+3}\\[1mm]
=((2k-2)a_1+c_3)\alpha_{2k-1}+\sum\limits_{t=3}^{k-1}(c_{2k-2t+3}+(2t-2)a_2\alpha_{2k-2t+3})\alpha_{2t-1}\\[1mm]
+\sum\limits_{t=3}^{k}(a_{2k-2t+3}+(2t-2)a_2\alpha_{2k-2t+4})\alpha_{2t-2},
\end{array}\right.
\end{equation}
\begin{equation}\label{eq6}4\leq k\leq \left\lfloor\frac{n}{2}\right\rfloor\left\{\begin{array}{cc}
(2a_1+a_2+c_3)\alpha_{2k-2}+\sum\limits_{t=5}^{2k-2}(c_{t-1}+a_2\alpha_{t-1})\alpha_{2k-t+2}\\[1mm]
=((2k-1)a_1+a_2)\alpha_{2k-2}+\sum\limits_{t=3}^{k-1}(c_{2k-2t+2}+(2t-2)a_2\alpha_{2k-2t+2})\alpha_{2t-1}+\\[1mm]
+\sum\limits_{t=3}^{k-1}(a_{2k-2t+2}+(2t-2)a_2\alpha_{2k-2t+3})\alpha_{2t-2}.
\end{array}\right.
\end{equation}

According to $b_1+b_2 = a_1 + a_2$, from equalities
\eqref{eq1} and \eqref{eq4} we obtain
\[(a_1-a_2)\alpha_4=0,\quad  (3a_1-c_3)\alpha_4=0.\]

Subtracting of equalities \eqref{eq2} and \eqref{eq5} we obtain
\[\sum\limits_{t=3}^{k}(a_{2k-2t+3}-c_{2k-2t+4}+a_2\alpha_{2k-2t+4})\alpha_{2t-2}=0, \quad 3 \leq k \leq \left\lfloor\frac{n-1} 2\right\rfloor.\]

Summarizing  of these  equalities  we get
\[\alpha_{k}(a_2-(k-3)a_1)=\frac{k-1}{2}a_2\sum\limits_{t=5}^{k}\alpha_{t-1}\alpha_{k-t+4}\quad \text{for odd }  k,\quad 5 \leq k \leq n-2.\]

Similarly from equalities \eqref{eq3} and \eqref{eq6} we obtain
\[(2a_1+a_2-c_3)\alpha_{2k}+\sum\limits_{t=3}^{k}(a_{2k-2t+4}-c_{2k-2t+5}+a_2\alpha_{2k-2t+5})\alpha_{2t-2}=0, \quad 3 \leq k \leq \left\lfloor\frac{n} 2\right\rfloor -1 \]
and
\[\alpha_{k}(a_2-(k-3)a_1)=\frac{k-1}{2}a_2\sum\limits_{t=5}^{k}\alpha_{t-1}\alpha_{k-t+4}\quad \text{for even }  k, \quad 5 \leq k \leq n-2.\]

From equalities \eqref{eq2} and \eqref{eq3} in the case of
$2k=n$ and $2k=n-1$, respectively we obtain last two restrictions
of equalities \eqref{eq0}.

Considering the properties of the pre-derivation for $P([[e_i,
e_1], e_2])$, $4 \leq i \leq n$ and $P([[e_i, e_2], e_2])$, $3
\leq i \leq n$, we have the same restrictions or identical
equalities.
\end{proof}

In the following proposition we give the descriptions of
pre-derivation of algebras from the second class of filiform
Leibniz algebras.
\begin{prop}\label{prop32}
The pre-derivations of the filiform Leibniz algebras from the
family $F_2(\beta_4,\beta_5,\ldots, \beta_{n}, \gamma)$ have the
following form:
\begin{align*}
P(e_1)&=\sum\limits_{t=1}^{n}a_te_t, \qquad
P(e_2)=b_2e_2+b_{n-1}e_{n-1}+b_ne_n, \qquad
P(e_3)=\sum\limits_{t=2}^{n}c_te_t,\\
P(e_{2i})&=(2i-1)a_1e_{2i}+\sum\limits_{t=2i+1}^{n}(a_{t-2i+2}+(2i-2)a_2\beta_{t-2i+3})e_t,&& 2\leq i\leq \left\lfloor\frac{n}{2}\right\rfloor,\\
P(e_{2i+1})&=((2i-2)a_1+c_3)e_{2i+1}+\sum\limits_{t=2i+2}^{n}(c_{t-2i+2}+(2i-2)a_2\beta_{t-2i+2})e_t, && 2\leq i\leq \left\lfloor\frac{n-1}{2}\right\rfloor,
\end{align*}
where
\begin{equation}\label{eq3.8} \left\{\begin{array}{ll}
(c_3-2a_1)\beta_4=0, \quad (b_2-2a_1)\beta_4=0,\quad
c_2\beta_t=0,& 4\leq
t\leq n-1,\\[1mm]
\sum\limits_{t=3}^{k}(a_{2k-2t+3}-c_{2k-2t+4}+a_2\beta_{2k-2t+4})\beta_{2t-2}=0, & 3\leq k\leq \lfloor\frac{n-1}{2}\rfloor,\\[1mm]
(c_3-2a_1)\beta_{2k}=\sum\limits_{t=3}^{k}(a_{2k-2t+4}-c_{2k-2t+5}+a_2\beta_{2k-2t+5})\beta_{2t-2},
& 3\leq
k\leq \lfloor\frac{n}{2}\rfloor-1,\\[1mm]
(b_2-(k-2)a_1)\beta_{k}=\frac{k-1}{2}a_2\sum\limits_{t=5}^{k}\beta_{t-1}\beta_{k-t+4}, &  5\leq k\leq n-2,\\[1mm]
(b_2-c_3-(n-5)a_1)\beta_{n-1}=a_2\sum\limits_{t=3}^{\frac{n-1}{2}}(2t-3)\beta_{n-2t+3}\beta_{2t-1}+\\[1mm]
\quad\quad\quad\quad\quad\quad\quad\quad
+\sum\limits_{t=2}^{\frac{n-3}{2}}(c_{n-2t+2}-a_{n-2t+1}+(2t-3)a_2\beta_{n-2t+2})\beta_{2t},
& n\quad \text {is odd,}\\[1mm]
(b_2-(n-3)a_1)\beta_{n-1}=a_2\sum\limits_{t=3}^{\frac{n-2}{2}}(2t-3)\beta_{n-2t+3}\beta_{2t-1}+\\[1mm]
\quad\quad\quad\quad\quad\quad\quad\quad+\sum\limits_{t=2}^{\frac{n-2}{2}}(c_{n-2t+2}-a_{n-2t+1}+(2t-3)a_2\beta_{n-2t+2})\beta_{2t},&
n\quad  \text {is even}.
\end{array}\right.
\end{equation}
\end{prop}

\begin{proof} Let $P$ be a pre-derivation of a filiform Leibniz algebra $L$ from the second class. Put
\[P(e_1)=\sum\limits_{t=1}^{n}a_te_t, \quad  P(e_2)=\sum\limits_{t=1}^{n}b_te_t, \quad P(e_3)=\sum\limits_{t=1}^{n}c_te_t.\]

Consider the property of pre-derivation
\begin{align*}
P([[e_2,e_1],e_1])&=[[P(e_2),e_1],e_1]+[[e_2,P(e_1)],e_1]+[[e_2,e_1],P(e_1)]\\
&=[[\sum\limits_{t=1}^{n}b_te_t,e_1],e_1]+[[e_2,\sum\limits_{t=1}^{n}a_te_t],e_1]\\
&=[b_1e_3+\sum\limits_{t=3}^{n-1}b_{t}e_{t+1},e_1]=b_1e_4+\sum\limits_{t=5}^{n}b_{t-2}e_{t}.
\end{align*}

On the other hand, $P([[e_2,e_1],e_1])=0$, since $[e_2, e_1]=0$.
Thus, we have \[ b_1=0, \quad b_t=0,  \quad 3\leq t\leq n-2.\]
Hence, $P(e_2)=b_2e_2+b_{n-1}e_{n-1}+b_ne_n$.

From
\begin{align*}
0&=P([[e_1,e_1],e_3])=[[P(e_1),e_1],e_3]+[[e_1,P(e_1)],e_3]+[[e_1,e_1],P(e_3)]\\
&=[e_3,\sum\limits_{t=1}^{n}c_te_t]=c_1e_4+c_2\sum\limits_{t=5}^{n}\beta_{t-1}e_t,
\end{align*}
we get \[c_1=0, \quad c_2\beta_t=0,\quad 4\leq t\leq n-1.\]


Considering the property of pre-derivation for the triples $\{e_1,
e_1, e_1\}$ and $\{e_i, e_1, e_1\}$ for $3 \leq i \leq n-2$,
inductively we obtain
\begin{align*}
P(e_{2i})&=(2i-1)a_1e_{2i}+\sum\limits_{t=2i+1}^{n}(a_{t-2i+2}+(2i-2)a_2\beta_{t-2i+3})e_t, && 2\leq i\leq \left\lfloor\frac{n}{2}\right\rfloor.\\
P(e_{2i+1})&=((2i-2)a_1+c_3)e_{2i+1}+\sum\limits_{t=2i+2}^{n}(c_{t-2i+2}+(2i-2)a_2\beta_{t-2i+2})e_t, &&  2\leq i\leq \left\lfloor\frac{n-1}{2}\right\rfloor.
\end{align*}

Now, we consider
\begin{align*}
P([[e_1,e_1],e_2])&=[[P(e_1),e_1],e_2]+[[e_1,P(e_1)],e_2]+[[e_1,e_1],P(e_2)]\\
&=[[\sum\limits_{t=1}^{n}a_te_t,e_1],e_2]+[[e_1,\sum\limits_{t=1}^{n}a_te_t],e_2]+[e_3,b_2e_2+b_{n-1}+b_te_t]\\
&=(2a_1+b_2)\beta_4e_5+\sum\limits_{k=3}^{\lfloor\frac{n}{2}\rfloor}\Big[(2a_1+b_2)\beta_{2k-1}
\sum\limits_{t=4}^{2k-2}(a_{t-1}+a_2\beta_t)\beta_{2k-t+2}\Big]e_{2k}\\
&\quad  +\sum\limits_{k=3}^{\lfloor\frac{n-1}{2}\rfloor}\Big[(2a_1+b_2)\beta_{2k}+\sum\limits_{t=4}^{2k-1}(a_{t-1}+a_2\beta_t)\beta_{2k-t+3}\Big]e_{2k+1}.
\end{align*}

On the other hand,
\begin{align*}
P([[e_1,e_1],e_2])&=P([e_3,e_2])=P\Big(\sum\limits_{t=5}^{n}\beta_{t-1}e_t\Big)\\
& =\beta_4(2a_1+c_3)e_5
+\sum\limits_{k=3}^{\lfloor\frac{n}{2}\rfloor}\Big[\beta_{2k-1}(2k-1)a_1  +\sum\limits_{t=2}^{k-1}\beta_{2t}(c_{2k-2t+2}+(2t-2)a_2\beta_{2k-2t+2})\\
 &  \qquad  \qquad \qquad \qquad \qquad +\sum\limits_{t=3}^{k-1}\beta_{2t-1}(a_{2k-2t+2}+(2t-2)a_2\beta_{2k-2t+3})\Big]e_{2k}\\
& \qquad +\sum\limits_{k=3}^{\lfloor\frac{n-1}{2}\rfloor}\Big[\beta_{2k}((2k-2)a_1+c_3)+
\sum\limits_{t=2}^{k-1}\beta_{2t}(c_{2k-2t+3}+(2t-2)a_2\beta_{2k-2t+3})\\
& \qquad  \qquad \qquad  +\sum\limits_{t=3}^{k}\beta_{2t-1}(a_{2k-2t+3}+(2t-2)a_2\beta_{2k-2t+4})\Big]e_{2k+1}.
\end{align*}
Comparing the coefficients at the basis elements, we have
\begin{equation}\label{eq3.9}(2a_1+b_2)\beta_4=(2a_1+c_3)\beta_4.\end{equation}
\begin{equation}\label{eq3.10}
3\leq k\leq \left\lfloor\frac{n}{2}\right\rfloor \left\{\begin{array}{c} (2a_1+b_2)\beta_{2k-1}+\sum\limits_{t=4}^{2k-2}(a_{t-1}+a_2\beta_t)\beta_{2k-t+2}\\[1mm]
=\beta_{2k-1}(2k-1)a_1+\sum\limits_{t=2}^{k-1}\beta_{2t}(c_{2k-2t+2}+(2t-2)a_2\beta_{2k-2t+2})\\[1mm]
+\sum\limits_{t=3}^{k-1}\beta_{2t-1}(a_{2k-2t+2}+(2t-2)a_2\beta_{2k-2t+3}),
\end{array}\right.\end{equation}
\begin{equation}\label{eq3.11}
3\leq k\leq \left\lfloor\frac{n-1}{2}\right\rfloor
\left\{\begin{array}{c}(2a_1+b_2)\beta_{2k}+\sum\limits_{t=4}^{2k-1}(a_{t-1}+a_2\beta_t)\beta_{2k-t+3}\\[1mm]
=\beta_{2k}((2k-2)a_1+c_3)+\sum\limits_{t=2}^{k-1}\beta_{2t}(c_{2k-2t+3}+(2t-2)a_2\beta_{2k-2t+3})\\[1mm]
+\sum\limits_{t=3}^{k}\beta_{2t-1}(a_{2k-2t+3}+(2t-2)a_2\beta_{2k-2t+4}).
\end{array}\right.
\end{equation}

Analogously, from the equality \[P([e_4, e_2]) =
P([[e_3,e_1],e_2])=[[P(e_3),e_1],e_2]+[[e_3,P(e_1)],e_2]+[[e_3,e_1],P(e_2)],\]
we obtain the following restrictions:
\begin{equation}\label{eq3.12}(a_1+b_2+c_3)\beta_{4}=5a_1\beta_4,\end{equation}
\begin{equation}\label{eq3.13}3\leq k\leq \left\lfloor\frac{n-1}{2}\right\rfloor\left\{\begin{array}{c}(a_1+b_2+c_3)\beta_{2k-1}+\sum\limits_{t=5}^{2k-1}(c_{t-1}+a_2\beta_{t-1})\beta_{2k-t+3}\\[1mm]
=\beta_{2k-1}((2k-2)a_1+c_3)+\sum\limits_{t=3}^{k-1}\beta_{2t-1}(c_{2k-2t+3}+(2t-2)a_2\beta_{2k-2t+3})\\[1mm]
+\sum\limits_{t=3}^{k}\beta_{2t-2}(a_{2k-2t+3}+(2t-2)a_2\beta_{2k-2t+4}),\end{array}\right.\end{equation}
\begin{equation}\label{eq3.14}4\leq k\leq \left\lfloor\frac{n}{2}\right\rfloor\left\{\begin{array}{c}(a_1+b_2+c_3)\beta_{2k-2}+\sum\limits_{t=5}^{2k-2}(c_{t-1}+a_2\beta_{t-1})\beta_{2k-t+2}\\[1mm]
=\beta_{2k-2}(2k-1)a_1+\sum\limits_{t=3}^{k-1}\beta_{2t-1}(c_{2k-2t+2}+(2t-2)a_2\beta_{2k-2t+2})\\[1mm]
+\sum\limits_{t=3}^{k-1}\beta_{2t-2}(a_{2k-2t+2}+(2t-2)a_2\beta_{2k-2t+3}).\end{array}\right.\end{equation}

It is not difficult to see that from \eqref{eq3.9} and
\eqref{eq3.12} we have
\[(c_3-2a_1)\beta_4=0, \quad (b_2-2a_1)\beta_4=0.\]

Similarly to the proof of Proposition \ref{prop1}, summarizing and
subtracting equalities \eqref{eq3.10} and \eqref{eq3.13}, we
obtain
\[\sum\limits_{t=3}^{k}(a_{2k-2t+3}-c_{2k-2t+4}+a_2\beta_{2k-2t+4})\beta_{2t-2}=0, \quad 3 \leq k \leq \left\lfloor\frac {n-1} 2\right\rfloor\]
and
\[(b_2-(k-2)a_1)\beta_{k}=\frac{k-1}{2}a_2\sum\limits_{t=5}^{k}\beta_{t-1}\beta_{k-t+4}\quad \text{for odd } k, \quad 5 \leq k \leq n-2.\]

From equalities \eqref{eq3.11} and \eqref{eq3.14} we have
\[(c_3-2a_1)\beta_{2k}=\sum\limits_{t=3}^{k}(a_{2k-2t+4}-c_{2k-2t+5}+a_2\beta_{2k-2t+5})\beta_{2t-2}, \quad 3 \leq k \leq \left\lfloor\frac {n} 2\right\rfloor -1,\] and
\[(b_2-(k-2)a_1)\beta_{k}=\frac{k-1}{2}a_2\sum\limits_{t=5}^{k}\beta_{t-1}\beta_{k-t+4}\quad \text{for even } k, \quad 5 \leq k \leq n-2.\]

From equalities \eqref{eq3.10} and \eqref{eq3.11} in the case
of $2k=n-1$ and $2k=n$, respectively we obtain last the two
restrictions of equalities \eqref{eq3.8}.

Considering the properties of the pre-derivation for $P([[e_i,
e_1], e_2])$ for $4 \leq i \leq n$ and $P([[e_i, e_2], e_2])$ for
$3 \leq i \leq n$, we have the same restrictions or identical
equalities.
\end{proof}

\section{Strongly nilpotent filiform Leibniz algebras}\label{S2}

In this section we describe non-strongly nilpotent filiform
Leibniz algebras. We give the description of non-strongly
nilpotent filiform Leibniz algebras of the first and second
classes. We reduce to investigation of strongly nilpotency of
third class of filiform Leibniz algebras to the investigation of Lie
algebras.

First, we consider the case of filiform Leibniz algebras from the
first class. From Proposition \ref{prop1} it is obvious that if
there exist the parameters $a_1, a_2, c_2, c_3$ such that $(a_1,
a_2, c_2, c_3) \neq (0, 0, 0, 0)$ and the restriction \eqref{eq0}
holds, then a filiform Leibniz algebra of the first family is
non-strongly nilpotent, otherwise is strongly nilpotent.

\begin{prop} \label{prop41} Let $L(\alpha_4, \alpha_5, \dots, \alpha_n, \theta)$ be a filiform Leibniz algebra of the first family. If
$\alpha_4 =  \alpha_5 =  \dots =\alpha_{n-1} =0$, then $L$ is
non-strongly nilpotent.
\end{prop}

\begin{proof} It is immediate that, if
$\alpha_4 =  \alpha_5 =  \dots =\alpha_{n-1} =0$, then the
restriction \eqref{eq0} is hold for any values of $a_1, a_2,
c_3$. Thus, we have that $L$ has a non-nilpotent pre-derivation,
which implies that $L$ is  non-strongly nilpotent.
\end{proof}

It is obvious that any algebra from the family $F_1(0,\dots, 0,
\alpha_n, \theta)$ is isomorphic to one of the following four
algebras:
\[F_1(0,\dots, 0, 0, 0), \quad F_1(0,\dots, 0, 0, 1),\quad F_1(0,\dots, 0, 1, 0), \quad F_1(0,\dots, 0, 1, 1).\]

Remark that the algebras $F_1(0,\dots, 0, 0, 0)$, $F_1(0,\dots, 0, 0,
1)$ and $ F_1(0,\dots, 0, 1, 1)$
 are non-characteristically nilpotent (see \cite{KhLO}). The algebra $F_1(0,\dots, 0, 1, 0)$ is characteristically nilpotent,
 but non-strongly nilpotent.

Now we consider the case of $\alpha_i\neq 0$ for some $i \ (4 \leq i
\leq n-1)$. Then from \eqref{eq0} we have that $c_2=0$.

\begin{thm}\label{thm41} Let $L$ be a filiform Leibniz
algebra from the family $F_1(\alpha_4, \alpha_5, \dots, \alpha_n,
\theta)$ and let $n$ be even.
 $L$ is non-strongly nilpotent if and only if the parameters $(\alpha_4, \alpha_5,  \alpha_6,\dots, \alpha_{n-1}, \alpha_n,
\theta)$ are one of the following values:

i) $\alpha_4 \neq 0$ and $\alpha_k =
(-1)^kC_{k-4}^2\alpha_4^{k-3}, \quad 5 \leq k \leq n-2$;

ii) $\left\{\begin{array}{lll}\alpha_{(2s-3)t+3} =
(-1)^{t+1}C_{t}^{2s-2}\alpha_{2s}^{t}, & 3 \leq s \leq \frac{n-2}2
& 1 \leq t \leq \lfloor\frac
{n-5} {2s-3}\rfloor, \\[1mm]
\alpha_j =0, & j \neq (2s-3)t+3, & 4 \leq j \leq
n-2;
\end{array}\right.$

iii) $\alpha_{2i}=0$, for $2 \leq i \leq \frac {n-2} 2$;

\noindent where $C^{p}_{n} = \frac {1} {(p-1)n+1} \dbinom{pn}{n}$ is the
$p$-th Catalan number.
\end{thm}

\begin{proof} From Proposition \ref{prop1} we have
\begin{equation}\label{eq4.1}\left\{\begin{array}{ll}
(a_1-a_2)\alpha_4=0,\quad  (3a_1-c_3)\alpha_4=0,\\[1mm]
\sum\limits_{t=3}^{k}(a_{2k-2t+3}-c_{2k-2t+4}+a_2\alpha_{2k-2t+4})\alpha_{2t-2}=0,&3\leq k\leq \frac{n-2}{2},\\[1mm]
(2a_1+a_2-c_3)\alpha_{2k}+\sum\limits_{t=3}^{k}(a_{2k-2t+4}-c_{2k-2t+5}+a_2\alpha_{2k-2t+5})\alpha_{2t-2}=0,& 3\leq k\leq \frac{n-2}{2},\\[1mm]
(a_2-(k-3)a_1)\alpha_{k}=\frac{k-1}{2}a_2\sum\limits_{t=5}^{k}\alpha_{t-1}\alpha_{k-t+4},& 5\leq k\leq n-2,\\[1mm]
(a_2-(n-4)a_1)\alpha_{n-1}=
a_2\sum\limits_{t=3}^{\frac{n-2}{2}}(2t-3)\alpha_{n-2t+3}\alpha_{2t-1}\\[1mm]
\quad\quad\quad\quad\quad\quad\quad\quad+\sum\limits_{t=2}^{\frac{n-2}{2}}(c_{n-2t+2}-a_{n-2t+1}+(2t-3)a_2\alpha_{n-2t+2})\alpha_{2t}.
\end{array}\right.
\end{equation}

\textbf{Case 1.} If $\alpha_4 \neq 0$, then from the first two
equalities of \eqref{eq4.1} we get $a_2 =a_1$, $c_3=3a_1$ and from
the next two equalities of \eqref{eq4.1} we obtain \[c_i = a_{i-1} + a_1 \alpha_i,
\qquad 4 \leq i \leq n-3.\]

Since $a_2 =a_1$, $c_3=3a_1$, we get that $L$ is non-strongly
nilpotent if and only if $a_1 \neq 0$. Therefore we have
\begin{align*}
\alpha_{k}&=\frac{k-1}{2(4-k)}\sum\limits_{t=5}^{k}\alpha_{t-1}\alpha_{k-t+4},
&&  5\leq k\leq n-2,\\
(5-n)a_1\alpha_{n-1}& = (c_{n-2}-a_{n-3} + a_1\alpha_{n-2})\alpha_4
+a_1\sum\limits_{t=5}^{n-2}(t-2)\alpha_{n-t+2}\alpha_{t}.
\end{align*}

Using equality \eqref{E:comb} we get that
\[\alpha_k =(-1)^kC_{k-4}^2\alpha_4^{k-3}, \qquad 5 \leq k \leq n-2\]
 and
\[c_{n-2} = \frac{1}{\alpha_4}\Big((5-n)a_1\alpha_{n-1} - a_1\sum\limits_{t=5}^{n-2}(t-2)\alpha_{n-t+2}\alpha_{t} \Big) +a_{n-3} - a_1\alpha_{n-2}.\]

Note that the parameters $\alpha_{n-1}, \alpha_{n}$ and $\theta$
are free and we have the case \textit{i}).

\textbf{Case 2.}  Let $\alpha_{2s} \neq 0$ for some $s \ (3 \leq s
\leq \frac {n-2} 2)$ and $\alpha_{2i} = 0$ for $2 \leq i \leq
s-1$. Then similarly to the previous case from  equality
\eqref{eq4.1} we get
\[(2a_1+a_2 - c_3)\alpha_{2s}=0, \qquad (a_2 - (2s-3) a_1)\alpha_{2s}=0,\]
which derive $a_2=(2s-3) a_1$, $c_3 = (2s-1) a_1$ and
\[c_i = a_{i-1} + a_1 \alpha_i, \qquad 4 \leq i \leq n-2s+1.\]

If $L$ is non-strongly nilpotent then $a_1 \neq 0$.  Consequently,
we have
\[\alpha_{k}=\frac{(k-1)(2s-3)}{2(2s-k)}\sum\limits_{t=5}^{k}\alpha_{t-1}\alpha_{k-t+4},\qquad  5\leq k\leq n-2.\]

Using $\alpha_{2i} = 0$ for $2 \leq i \leq s-1$ and applying
formula \eqref{E:comb}, inductively on $t$ we obtain
\begin{align*}
\alpha_j &= 0, \qquad   \qquad  j \neq (2s-3)t+3, \qquad  4 \leq j \leq n-2,\\
\alpha_{(2s-3)t+3} &= (-1)^{t+1}C_{t}^{2s-2}\alpha_{2s}^{t},  \qquad   \qquad   \qquad  \quad 1 \leq t \leq \left\lfloor\frac {n-5} {2s-3}\right\rfloor.
\end{align*}
From the last equality of \eqref{eq4.1} we have
\[c_{n-2s+2} = \frac{1}{\alpha_{2s}}\Big((2s+1-n)a_1\alpha_{n-1} - (2s-3)a_1\sum\limits_{t=2s+1}^{n-2s+2}(t-2)\alpha_{n-t+2}\alpha_{t} \Big) +a_{n-2s+1} -
(2s-3)^2a_1\alpha_{n-2s}.\]

The parameters $\alpha_n-1$, $\alpha_n$ and $\theta$ may take any
values and we obtain case \textit{ii}).

 \textbf{Case 3.} Let $\alpha_{2i}=0$
for all $i \ (2 \leq i \leq \frac {n-2} 2)$. Then the first four
equalities of \eqref{eq4.1} hold and from the last equalities
we have
\begin{align*}
&\alpha_5(a_2-2a_1) =0, \qquad \alpha_{2s+1}(a_2-(2s-2)a_1) =sa_2\sum\limits_{t=3}^s\alpha_{2t-1}\alpha_{2s+5-2t}, \quad 3 \leq s \leq \frac {n-2} 2.\\
&(a_2-(n-4)a_1)\alpha_{n-1}=
a_2\sum\limits_{t=3}^{\frac{n-2}{2}}(2t-3)\alpha_{n-2t+3}\alpha_{2t-1}.
\end{align*}

Taking $a_1=a_2=0$ and $c_3\neq 0$, we have that previous
equalities hold for any values of $\alpha_{2s+1}$. Since
$c_3\neq 0$, this algebra is non-strongly nilpotent and we
obtain the case \textit{iii}).
\end{proof}

\begin{thm}\label{thm42}
Let $L$ be a filiform Leibniz
algebra from the family $F_1(\alpha_4, \alpha_5, \dots, \alpha_n,
\theta)$ and let $n$ be odd.
 $L$ is non-strongly nilpotent if and only if the parameters $(\alpha_4, \alpha_5,  \alpha_6,\dots, \alpha_{n-1}, \alpha_n,
\theta)$ are one of the following values:

i) $\alpha_4 \neq 0$ and $\alpha_k =
(-1)^kC_{k-4}^2\alpha_4^{k-3}, \quad 5 \leq k \leq n-2$;

ii) $\left\{\begin{array}{lll}\alpha_{(2s-3)t+3} =
(-1)^{t+1}C_{t}^{2s-2}\alpha_{2s}^{t}, & 3 \leq s \leq \frac {n-3}
2,& 1 \leq t \leq \lfloor\frac {n-5} {2s-3}\rfloor,\\[1mm] \alpha_j = 0,
& j \neq (2s-3)t+2, & 4 \leq j \leq n-2;\end{array}\right.$

iii) $\alpha_{2i}=0$, for $2 \leq i \leq \frac {n-1} 2$;

iv) $\left\{\begin{array}{lll}\alpha_{(2s-2)t+3} =
(-1)^{t+1}C_{t}^{2s-1}\alpha_{2s+1}^{t}, & 2 \leq s \leq \frac
{n-3} 2, & 1 \leq t \leq \lfloor\frac {n-5} {2s-2}\rfloor, \\[1mm] \alpha_j
= 0, & j \neq (2s-2)t+3, & 4 \leq j \leq n-2.
\end{array}\right.$
\end{thm}

\begin{proof} From Proposition \ref{prop1}
we have

\begin{equation}\label{eq4.5}
\left\{\begin{aligned}
& (a_1-a_2)\alpha_4=0,\qquad \qquad  (3a_1-c_3)\alpha_4=0,\\
& \sum\limits_{t=3}^{k}(a_{2k-2t+3}-c_{2k-2t+4}+a_2\alpha_{2k-2t+4})\alpha_{2t-2}=0, \qquad \qquad \qquad \qquad \qquad  \quad  3\leq k\leq \frac{n-1}{2},\\
&(2a_1+a_2-c_3)\alpha_{2k}+\sum\limits_{t=3}^{k}(a_{2k-2t+4}-c_{2k-2t+5}+a_2\alpha_{2k-2t+5})\alpha_{2t-2}=0, \qquad 3\leq k\leq \frac{n-3}{2},\\
&(a_2-(k-3)a_1)\alpha_{k}=\frac{k-1}{2}a_2\sum\limits_{t=5}^{k}\alpha_{t-1}\alpha_{k-t+4}, \qquad \qquad \qquad \qquad \qquad \qquad \quad   5\leq k\leq n-2,\\
& (2a_2-c_3-(n-6)a_1)\alpha_{n-1}=a_2\sum\limits_{t=3}^{\frac{n-1}{2}}(2t-3)\alpha_{n-2t+3}\alpha_{2t-1}\\
& \qquad \qquad \qquad \qquad \qquad \qquad \quad  +\sum\limits_{t=2}^{\frac{n-3}{2}}(c_{n-2t+2}-a_{n-2t+1}+(2t-3)a_2\alpha_{n-2t+2})\alpha_{2t}
\end{aligned}\right.
\end{equation}
\quad\quad\quad\quad\qquad\qquad

\textbf{Case 1.}  $\alpha_{2s} \neq 0$ for some $s \ (2 \leq s \leq
\frac {n-3} 2)$ and  $\alpha_{2i} = 0$ for $2 \leq i \leq s-1$.
 Then similarly to the proof of Theorem \ref{thm41}  we get
\[a_2=(2s-3) a_1,\qquad c_3 = (2s-1) a_1, \qquad c_i = a_{i-1} + a_1 \alpha_i, \qquad 4 \leq i \leq n-2s+1.\]

 Consequently
we have
\[\alpha_{k}=\frac{(k-1)(2s-3)}{2(2s-k)}\sum\limits_{t=5}^{k}\alpha_{t-1}\alpha_{k-t+4},\quad  5\leq k\leq n-2.\]

Applying formula \eqref{E:comb} inductively on $t$ we obtain
\[\alpha_{(2s-3)t+3} = (-1)^{t+1}C_{t}^{2s-2}\alpha_{2s}^{t}, \qquad 1 \leq t \leq \left\lfloor\frac {n-5} {2s-3}\right\rfloor,\]
\[\alpha_j = 0, \quad j \neq (2s-3)t+3, \quad  4 \leq j \leq n-2.\]

From the last equality of \eqref{eq4.5} we have
\[c_{n-2s+2} = \frac{1}{\alpha_{2s}}\Big((2s+1-n)a_1\alpha_{n-1} - (2s-3)a_1\sum\limits_{t=2s+1}^{n-2s+2}(t-2)\alpha_{n-t+2}\alpha_{t} \Big) +a_{n-2s+1} -
(2s-3)^2a_1\alpha_{n-2s}.\]

Thus, we have the cases \textit{i}) and \textit{ii}).

\textbf{Case 2.} Let $\alpha_{2i}=0$ for all $i \ (2 \leq i \leq
\frac {n-3} 2)$. Then the first four equalities of \eqref{eq4.5}
hold and from the last two equalities we have
\begin{equation}\label{eq.forodd}
\begin{aligned}
&\alpha_5(a_2-2a_1) =0, \qquad \alpha_{2s+1}(a_2-(2s-2)a_1)
=sa_2\sum\limits_{t=3}^s\alpha_{2t-1}\alpha_{2s+5-2t}, \quad 3
\leq s \leq \frac {n-3} 2. \\
&(2a_2-c_3-(n-6)a_1)\alpha_{n-1}=0.
\end{aligned}
\end{equation}

If $\alpha_{n-1}=0$, then taking $a_1=a_2=0$ and $c_3 \neq 0$, we
have a non-nilpotent pre-derivation for any values of
$\alpha_{2i+1}$, and so we have the case \textit{iii}).

If $\alpha_{n-1}\neq0$, then $c_3  = 2a_2+(n-6)a_1$ and we obtain
that

 non-nilpotency of pre-derivations depends of the parameters
$\alpha_{2i+1}$.

Let $\alpha_{2s+1}$ be the first non-vanishing parameter among
$\{\alpha_5,\alpha_7, \dots, \alpha_{n-4},\alpha_{n-2}\}$. Then we
get \[a_2 = (2s-2)a_1.\]

Applying formula \eqref{E:comb} inductively on $t$ from
\eqref{eq.forodd} we obtain
\[\alpha_{(2s-2)t+3} = (-1)^{t+1}C_{t}^{2s-2}\alpha_{2s}^{t}, \qquad 1 \leq t \leq \left\lfloor\frac {n-5} {2s-3}\right\rfloor\]
 and
\[\alpha_j = 0, \qquad j \neq (2s-2)t+3, \quad 4 \leq j \leq n-2.\]

Therefore, we have the case \textit{iv}).
\end{proof}

Now we give the classification of non-strongly nilpotent filiform
Leibniz algebras from the second class.

\begin{prop}  Let $L(\beta_4, \beta_5, \dots, \beta_n, \gamma)$ be a filiform Leibniz algebra of the second family. If
$\beta_4 =  \beta_5 =  \dots =\beta_{n-1} =0$, then $L$ is
non-strongly nilpotent.
\end{prop}

\begin{proof} Analogously to the proof of Proposition \ref{prop41}.
\end{proof}

It is obvious that any algebra from the family $F_2(0,\dots, 0,
\beta_n, \gamma)$ is isomorphic to one of the algebras
$F_2(0,\dots, 0, 0, 0)$, $ F_2(0,\dots, 0, 1, 0)$, $F_2(0,\dots,
0, 0, 1)$. Note that these algebras are non-characteristically
nilpotent (see \cite{KhLO}).

Now we consider the case of $\beta_i\neq 0$ for some $i \ (4 \leq i
\leq n-1)$. Then from \eqref{eq0} we have that $c_2=0$.

\begin{thm}\label{SC.n.even} Let $L$ be a $n$-dimensional complex  non-strongly nilpotent filiform Leibniz algebra from the family
$F_2(\beta_4,\ldots,\beta_n,\gamma)$ and $n$  even. Then $L$ is
isomorphic to one of the following algebras:
\begin{align*}
& F_2^{2s}(0,\dots, 0, \beta_{2s}, 0 \dots, 0 ,
\beta_{n-1},\beta_n,\gamma), \qquad \beta_{2s}=1, \quad 2 \leq s
\leq \frac{n-2} 2,\\
& F_2(0, \beta_5, 0, \beta_7, 0 , \dots, 0, \beta_{n-1}, \beta_n,
\gamma).
\end{align*}
\end{thm}

\begin{proof}
 From Proposition \ref{prop32} we have:

\begin{equation}\label{eq4.8}
 \left\{\begin{array}{ll}
(c_3-2a_1)\beta_4=0, \quad (b_2-2a_1)\beta_4=0,& \\[1mm]
\sum\limits_{t=3}^{k}(a_{2k-2t+3}-c_{2k-2t+4}+a_2\beta_{2k-2t+4})\beta_{2t-2}=0, & 3\leq k\leq \frac{n-2}{2},\\[1mm]
(c_3-2a_1)\beta_{2k}=\sum\limits_{t=3}^{k}(a_{2k-2t+4}-c_{2k-2t+5}+a_2\beta_{2k-2t+5})\beta_{2t-2},
& 3\leq
k\leq \frac{n-2}{2},\\[1mm]
(b_2-(k-2)a_1)\beta_{k}=\frac{k-1}{2}a_2\sum\limits_{t=5}^{k}\beta_{t-1}\beta_{k-t+4}, &  5\leq k\leq n-2,\\[1mm]
(b_2-(n-3)a_1)\beta_{n-1}=a_2\sum\limits_{t=3}^{\frac{n-2}{2}}(2t-3)\beta_{n-2t+3}\beta_{2t-1}\\[1mm]
\qquad\qquad\qquad\qquad\qquad \ \ +
\sum\limits_{t=2}^{\frac{n-2}{2}}(c_{n-2t+2}-a_{n-2t+1}+(2t-3)a_2\beta_{n-2t+2})\beta_{2t},
\end{array}\right.
\end{equation}

\textbf{Case 1.} Let $\beta_4\neq 0$, then from \eqref{eq4.8} we
have
\begin{align*}
 c_3&=b_2=2a_1, \qquad c_{i} =a_{i-1}+a_2\beta_i,\qquad \quad  4 \leq i \leq n-3,\\
(4-k)a_1\beta_{k}&=\frac{k-1}{2}a_2\sum\limits_{t=5}^{k}\beta_{t-1}\beta_{k-t+4},\qquad \qquad \qquad  5\leq k\leq n-2,\\
(5-n)a_1\beta_{n-1}&=(c_{n-2}-a_{n-3}+a_2\beta_{n-2})\beta_{4}+a_2\sum\limits_{t=5}^{n-2}(t-2)\beta_{n-t+2}\beta_t.
\end{align*}

Since $L$ is non-strongly nilpotent, we get $a_1\neq0$ and
\[
\beta_5=-\frac{2a_2}{a_1}\beta_4^2, \qquad
\beta_{k}=\frac{(k-1)a_2}{2(4-k)a_1}\sum\limits_{t=5}^{k}\beta_{t-1}\beta_{k-t+4},\quad
6\leq k\leq n-2.
\]

From the isomorphic criterion which given in [10, Theorem 4.4] we
have that if two algebras from the family
$F_2(\beta_4,\ldots,\beta_n,\gamma)$ are isomorphic,  then there
exist $A,B,D \in \mathbb{C}$, such that
\[\beta_4'=\frac{D}{A^2}\beta_4,\qquad \beta_5'=\frac{D}{A^3}(\beta_5-\frac{2B}{A}\beta_4^2),\]
where $\beta_i$ and $\beta_i'$ are parameters of the first and
second algebras, respectively.

Putting $D=\frac{A^2}{\beta_4}$ and
$B=\frac{A\beta_5}{2\beta_4^2}$ we obtain $\beta_4'=1$,
$\beta_5'=0$. Therefore, we have shown that, if $L$ is
non-strongly nilpotent algebra from the family
$F_2(\beta_4,\ldots,\beta_n,\gamma)$, with $\beta_4 \neq 0$, then
we may always suppose
\[\beta_4=1, \qquad \beta_5=0.\]

Moreover, from  $\beta_5=-\frac{2a_2}{a_1}\beta_4^2$, we obtain
$a_2=0$, which implies $\beta_k=0$ for $6\leq k\leq n-2$ and
\[c_{n-2} = a_{n-3} + \frac{(5-n)a_1\beta_{n-1}}{\beta_4}.\]

Thus, $L$ isomorphic to the algebra
$F_2(1,0,\ldots,0,\beta_{n-1},\beta_n,\gamma)$.

\textbf{Case 2.} Let $\beta_{2s}\neq 0$ for some $s \ (3 \leq s \leq
\frac{n-2} 2)$ and $\beta_{2i}=0$ for $2 \leq i \leq s-1$. Then we
have
\begin{align*}
b_2&=(2s-2)a_1, \quad c_3 = 2a_1, \quad c_{i} =a_{i-1}+a_2\beta_i,\quad  4 \leq i \leq n-2s+1,\\
(2s-k)a_1\beta_{k}& =\frac{k-1}{2}a_2\sum\limits_{t=5}^{k}\beta_{t-1}\beta_{k-t+4},\qquad \qquad \qquad \qquad  \quad  5\leq k\leq n-2,\\
(2s+1-n)a_1\beta_{n-1}&=(c_{n-2s+2}-a_{n-2s+1}+a_2\beta_{n-2s+2})\beta_{2s}+a_2\sum\limits_{t=5}^{n-2}(t-2)\beta_{n-t+2}\beta_t.
\end{align*}

Since $L$ is non-strongly nilpotent, we get $a_1\neq0$, which
implies $\beta_5=\dots=\beta_{2s-1}=0$. Moreover, we have
 \[\beta_{4s-3} =
-\frac{(2s-2)a_2}{(2s-3)a_1}\beta_{2s}^2.\]

From the isomorphic criterion of filiform Leibniz algebras of the
second class, we obtain
\[\beta_{2s}'=\frac{D}{A^{2s-2}}\beta_{2s},\qquad \beta_{4s-3}'=\frac{D}{A^{4s-2}}\left(\beta_{4s-3}-\frac{(2s-2)B}{A}\beta_{2s}^2\right).\]

Putting $D=\frac{A^{2s-2}}{\beta_{2s}}$ and
$B=\frac{A\beta_{4s-3}}{(2s-2)\beta_{2s}^2}$, we have
$\beta_{2s}'=1$, $\beta_{4s-3}'=0$. Therefore, we have shown that,
if $L$ is non-strongly nilpotent algebra from the family
$F_2(\beta_4,\ldots,\beta_n,\gamma)$, with $\beta_{2s} \neq 0$ and
$\beta_{2i} = 0$ for $2 \leq i \leq s-1$, then we may always
suppose
\[\beta_{2s}=1, \quad \beta_{4s-3}=0.\]

Moreover, from  $\beta_{4s-3} =
-\frac{(2s-2)a_2}{(2s-3)a_1}\beta_{2s}^2$, we obtain $a_2=0$,
which implies $\beta_k=0$ for $2s+1\leq k\leq n-2$ and
\[c_{n-2s+2} = a_{n-2s+1} + \frac{(2s+1-n)a_1\beta_{n-1}}{\beta_{2s}}.\]

Thus, $L$ isomorphic to the algebra \[F_2^{2s}(0,\dots, 0,
\beta_{2s}, 0 \dots, 0 , \beta_{n-1},\beta_n,\gamma), \quad
\beta_{2s}=1.\]

\textbf{Case 3.} Let $\beta_{2s}= 0$ for $2 \leq s \leq \frac
{n-2} 2$, then we have
\begin{align*}
(b_2-(k-2)a_1)\beta_{k}& =\frac{k-1}{2}a_2\sum\limits_{t=5}^{k}\beta_{t-1}\beta_{k-t+4}, &&  5\leq k\leq n-2,\\
(b_2-(n-3)a_1)\beta_{n-1}& =a_2\sum\limits_{t=3}^{\frac{n-2}{2}}(2t-3)\beta_{n-2t+3}\beta_{2t-1}.
\end{align*}

Taking $a_1=a_2=0$ and $c_3\neq 0$, we have that previous
equalities are hold for any values of $\beta_{2s+1}$. Since
$c_3\neq 0$, this algebra is non-strongly nilpotent. Therefore, we
obtain the algebra $F_2(0, \beta_5, 0, \beta_7, 0 , \dots, 0,
\beta_{n-1}, \beta_n, \gamma)$.
\end{proof}

\begin{thm} Let $L$ be a $n$-dimensional complex  non-strongly nilpotent filiform Leibniz algebra from the family
$F_2(\beta_4,\ldots,\beta_n,\gamma)$ and $n$ odd. Then $L$ is
isomorphic to one of the following algebras:
\begin{align*}
&F_2^{j}(0,\dots, 0, \beta_{j}, 0 \dots, 0 ,
\beta_{n-1},\beta_n,\gamma), \qquad \beta_{j}=1, \quad 4 \leq j
\leq n-2,\\
& F_2(0, \beta_5, 0, \beta_7, 0 , \dots, 0, \beta_{n-2},0,  \beta_n, \gamma).
\end{align*}
\end{thm}

\begin{proof}
Analogously to the proofs of Theorems \ref{thm42} and
\ref{SC.n.even}.
\end{proof}

Now, we consider a Leibniz algebra $L$ from the third family
$F_3(\theta_1,\theta_2,\theta_3)$. Note that $L$ is a parametric
algebra with parameters $\theta_1,\theta_2,\theta_3$ and
$\alpha_{i,j}^k$. The parameters $\alpha_{i,j}^k$ appear from the
multiplications $[e_i, e_j]$ for $\leq i < j \leq {n-1}$. In the
case of  $\theta_1 = \theta_2 = \theta_3 =0$, the algebra $L$ is a Lie
algebra.

In the next proposition we assert that the strongly nilpotency of
$L(\theta_1,\theta_2,\theta_3)$ is equivalent to the strongly
nilpotency of $L(0,0,0)$.

\begin{prop} An algebra $L(\theta_1,\theta_2,\theta_3)$ from the family $F_3(\theta_1,\theta_2,\theta_3)$ is strongly nilpotent if and only if the algebra
$L(0,0,0)$ is strongly nilpotent.
\end{prop}
\begin{proof} Note that the parameters $\theta_1,\theta_2,\theta_3$ appear only in the multiplications $[e_1,e_1], [e_1,e_2], [e_2,e_2]$.
Since $[P(x),y],\ [x,P(y)], \ [x,y]\in L^2$ and $e_1,e_2\not \in
L^2$ parameters $\theta_i$ do not take part in the identity
\[P([[x,y],z])=[[P(x),y],z]+[[x,P(y)],z]+[[x,y],P(z)],\quad x,y,z \in L.\]

Therefore, the spaces of pre-derivations of algebras
$L(\theta_1,\theta_2,\theta_3)$ and $L(0,0,0)$  coincide.
\end{proof}

\textbf{Acknowledgments.} This work was supported by Agencia Estatal de Investigaci\'on (Spain), grant MTM2016-79661-P (European FEDER support included, UE).
 The second named author was supported
by IMU/CDC Grant (Abel Visiting Scholar Program), and he would
like to acknowledge the hospitality of the University of Santiago
de Compostela (Spain).

\end{document}